\documentclass[a4 paper,12pt]{article}
\usepackage{tikz}
\usepackage{amsmath,amsthm,amssymb,amscd,latexsym}
\usetikzlibrary{intersections}
\begin{document}

\begin{center}
\section*{Pizza Race Problem}

Keyue Gao\footnote{Keyue Gao's research supported by NYU's Summer Undergraduate Research Experience(SURE) program}

\end{center}

\begin{abstract}
This paper deals with a problem in which two players share a previously sliced pizza and try to eat as much amount of pizza as they can. It takes time to eat each piece of pizza and both players eat pizza at the same rate. One is allowed to take a next piece only after the person has finished eating the piece on hand. Also, after the first piece is taken, one can only take a piece which is adjacent to already-taken piece. This paper shows that, in this real time setting, the starting player can always eat at least $\frac{2}{5}$ of the total size of the pizza. However, this may not be the best possible amount the starting player can eat. It is a modified problem from an original one where two players takes piece alternatively instead. 
\end{abstract}

\begin{center}
\section{Introduction}
\end{center}

The original problem was first proposed by Peter Winkler \cite{Pe} at ``Building Bridges: a conference on mathematics and computer science in honour of Laci Lov´as,'' in Budapest, August 5-9 2008.  Alice and Bob share a pizza. The pizza is sliced into any number of pieces and each piece with arbitrary size. Alice starts by picking any piece she wants. After that, they pick piece alternatively and they can only take one of the two pieces near the cut. If the total number of piece is even, it can be easily shown that Alice can always get more than $\frac{1}{2}$ of the whole pizza. The challenge lies in the odd-number case. Peter Winkler had previously discovered a configuration with fifteen pieces in which Alice can get no more than $\frac{4}{9}$ of the whole pizza and he conjectured at the conference that there is a strategy which always guarantees Alice more than $\frac{4}{9}$ regardless of different ways of cutting the pizza. 

This conjecture was proved correct in the paper \emph{How to Eat $4/9$ of a Pizza} authored by Kolja Knauer, Piotr Micek and Torsten Ueckerdt \cite{KMU}. In this paper, we will put in an extra parameter: time to modify the orginal problem. Assume that both players spend some time to consume each piece and they are allowed to pick the next piece only after they have done eating. The formal and complete rules are given as follows:
\begin{enumerate}
\item Alice starts by picking any piece she wants. 
\item Afterwards, both players can only take pieces adjacent to the already-eaten pieces. In other words, they always take one of the two available pieces, except for the first and the last piece.
\item Both of them eat the pizza at the same rate. One is allowed to take the next piece only after he or she finishes eating the piece on hand. Therefore, the time it takes to eat one piece of pizza is proportional to its size. The time of taking pizza from the box is negligible. And we assume that after Alice takes her first piece, Bob \emph{immediately} makes a decision and takes one of two pieces he favors.
\item The weights of all pieces are chosen so that both players never finish eating piece at exactly the same time.
\end{enumerate} 
\textbf{Question}: Is there a constant $\alpha \in (0,1)$ such that Alice can eat at least $\alpha$ of the whole pizza regardless of the distribution? If so, what is the best possible $\alpha$?

Clearly $\alpha$ cannot exceed $\frac{1}{2}$. If we make a pizza with two pieces whose difference is arbitrarily small (remember the pieces are not allowed to have the same size), then Alice can never get more than $\frac{1}{2}$. This paper gives a proof that $\alpha>\frac{2}{5}$, but $\frac{2}{5}$ is not necessarily the best possible bound. When the number of pieces is small as 3 or 4, it can be showed Alice can always get more than $\frac{1}{2}$. However, the analysis quickly becomes complicated as the number of pieces increases. Future efforts can be made to find a refiner bound and show that it is the best possible. 

\begin{center}
\section{The proof}
\end{center}

Although the goal is to prove that Alice can get more than $\alpha$ (here we assume that $\alpha \leq \frac{1}{2}$), it is more convenient to examine what happens when she gets less than $\alpha$, or equivalently, Bob gets more than $1-\alpha$. 

Suppose we have found a division of a pizza into n pieces such that Bob has a strategy guarantees him at least $1-\alpha$ of the pizza no matter what Alice does. We label the n pieces from $P_1$ to $P_n$ (which can be also regarded as $P_0$) counterclockwise. Consider Figure \ref{npizza}.

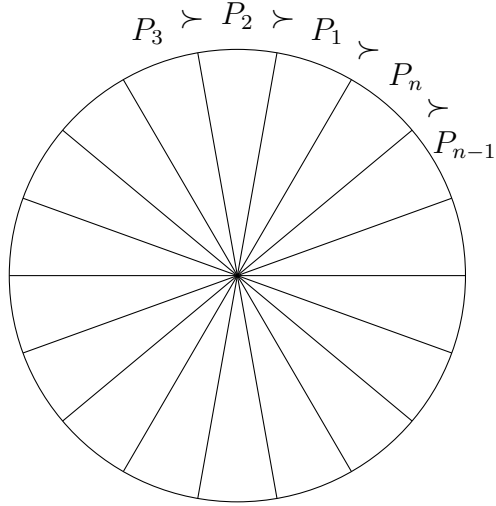
\begin{figure}
\begin{center}
\begin{tikzpicture}[scale=3]
  \draw (0,0) circle (1 cm);
  \foreach \x in {20,40,60,...,360} 
    \draw (0,0) -- (\x:1cm);
    \node at (70:1.15cm) {$P_1$}; 
     \node at (90:1.15cm) {$P_2$};
     \node at (110:1.15cm) {$P_3$};
     \node at (50:1.15cm) {$P_n$};
     \node at (30:1.15cm) {$P_{n-1}$};
   \foreach \x in {40,60,80,100}
     \node at (\x:1.15cm) {$\succ$};
    \foreach \x in {250,270,290}
     \node at (\x:1.15cm) {.};
\end{tikzpicture}
\end{center}
\caption{Bob's favorite pizza. \label{npizza}}
\end{figure}

 If Alice picks $P_1$ at the beginning, by assumption, Bob has a winning strategy which guarantees him more than $1-\alpha$. WLOG, let the winning strategy say that Bob should pick $P_2$. Now, if Alice changes her mind and switches to $P_2$ at the beginning, Bob \emph{cannot} choose $P_1$ because Alice can use the exact same winning strategy which Bob  uses just now against Bob and get more than $1-\alpha$ which is certainly more than $\alpha$. Therefore, the only choice left for Bob is to pick $P_3$ if Bob can win. The same arguement shows that if Alice picks $P_i$ initially, then Bob has to pick $P_{i+1}$. Note that here the assumption that both players never finish eating at the exactly the same time is crucial, because this assumption guarantees that Alice is able to completely copy Bob's strategy. Let's formalize this observation by introducing some notation. 

Let $P_i$ and $P_{i+1}$ be two adjacent pieces, we call $P_i \prec P_{i+1}$ if after two players pick $P_i$ and $P_{i+1}$ at the beginning, there is a strategy ensures that the person who gets $P_{i+1}$ will get more than $1-\alpha$ eventually. And we call that $P_i \sim P_{i+1}$ if both of them, playing their best, get more than $\alpha$ and less than $1-\alpha$. $P_i \preceq P_{i+1}$ means either $P_i \prec P_{i+1}$ or $P_i \sim P_{i+1}$. Since Alice's goal is get more than $\alpha$, it's enough for her to pick a piece $P_i$ such that $P_i \succeq P_{i+1}$ and $P_{i} \succeq P_{i-1}$. The previous observation tells us that if Bob can get more than $1-\alpha$, then it must be that $P_1 \prec P_2, P_2 \prec P_3,...,P_n \prec P_1$ or $P_1 \succ P_2, P_2 \succ P_3,...,P_n \succ P_1$. Now if we want to know when Alice can get more than $\alpha$, we simply take the negation of the \emph{principle of uniform direction} and it will be summarised and proved rigidly in the following key lemma. 

\newtheorem*{Lemma}{Lemma}
\begin{Lemma}
Given a pizza with n pieces from $P_1$ to $P_n$ and a constant $\alpha \leq \frac{1}{2}$, Alice can get more than $\alpha$ if there exist two pair of pieces $P_i, P_{i+1}$ and $P_j, P_{j+1}$ such that $P_i \preceq P_{i+1}$ but $P_j \succeq P_{j+1}$
\end{Lemma}
\begin{proof}
WLOG, assume $j \geq i$. Let's examine the relation between $P_{i+1}$ and $P_{i+2}$. If $P_{i+1} \succeq P_{i+2}$, then Alice should pick $P_{i+1}$ at the beginning. If neither is true, i.e. $P_{i+1} \prec P_{i+2}$, then continue the previous process to check whether $P_{i+2} \succeq P_{i+3}$ or $P_{i+2} \prec P_{i+3}$. Note that the direction of preference cannot be uniform all the way to $P_j$ because $P_j \succeq P_{j+1}$. Some where between $P_i$ and $P_{j+1}$, there is a $P_{k}$ such that $P_k \succ P_{k-1}$ and $P_k \succeq P_{k+1}$.
\end{proof}
Remark: Notice here we allow $i=j$ which means that $P_i \preceq P_{i+1}$ and $P_{i} \succeq P_{i+1}$, so it must be true that $P_{i} \sim P_{i+1}$. A quick corollary of the lemma is that Alice can get more than $\alpha$ if we can find one pair of adjacent pieces $P_{i}$ and $P_{i+1}$ such that $P_{i} \sim P_{i+1}$.

Using this lemma, we can prove the following theorem.
\newtheorem*{theorem}{Theorem}
\begin{theorem}
Alice can eat at least $\frac{2}{5}$ of the whole pizza regardless how the pizza is cut.
\end{theorem}
\begin{proof}
Suppose there exists a configuration such that Alice gets less than $\frac{2}{5}$ and Bob more than $\frac{3}{5}$. The difference is more than $\frac{1}{5}$. Since Alice and Bob are eating at the same time and at the same rate, the amount of pizza they eat at the beginning is always the same. The difference is created only in the end when one player is still eating, but the other player has nothing to eat. To be precise, the difference is exactly the remaining amount of piece in the winner's hand at the instant when the loser finishes his or her last piece and finds no peice left on the table. Hence, in order to create a difference bigger than $\frac{1}{5}$, there must exist at least one piece whose size is strictly bigger than $\frac{1}{5}$. Furthermore, this big piece is the last one Bob picks and at the instant when Bob picks it, the difference between this big piece and the remainning cluster of pieces must be greater than $\frac{1}{5}$. Also, no piece is bigger than $\frac{2}{5}$ since otherwise Alice can just pick that at the beginning. Clearly, it is impossible to have five such big pieces. Hence, it is suffice to show that Bob cannot get more than $\frac{3}{5}$ if there are 1, 2, 3 or 4 pieces with size more than $\frac{1}{5}$.

\begin{itemize}
\item If there is only one such big piece, then Alice just picks that one at the beginning. 

\item If there are two, label them $A$ and $B$ and  call the size of them $a$ and $b$ respectively. WLOG assume $a<b$. Also, label the remaining two cluster of pieces $C$ and $D$ which have size $c$ and $d$. We consider cases according to their relative size. Consider Figure \ref{2pizza}.

\textbf{CASE I}: $c=0$ i.e. $A$ and $B$ are adjacent 

If both players pick $A$ and $B$ at the beginning, then one one can get more than $\frac{3}{5}$, in other words, $A \sim B$. The quick corollary of the lemma applies.

\textbf{CASE II}: $0<c<a$ 

Since $c \ne 0$, let $C_1$ and $C_2$ be the pieces in $C$ closest to $A$ and $B$ respectively. If the initial two pieces chosen are $A$ and $C_1$, then whoever claims $C_1$ can finish the entire $C$ and gets to $B$ before the other player finishes $A$, which means $C_1 \succeq A$. Similarly, $C_2 \preceq B$, the lemma applies.

\textbf{CASE III}: $b>d$ and $c>0$

Let $D_1$ be the piece in $D$ closest to $B$. Since $b>d$ and $d>c$, $b>c$. Therefore, if Alice starts by choosing $C_2$, the piece in C closest to B, or $D_1$, the piece in D closest to B, and Bob then picks $B$, Alice can get to $A$ before Bob finishes $B$ so that Bob cannot create a difference more than $\frac{1}{5}$. Therefore, $B \prec D_1$ and $B \prec C_1$. Again the lemma applies.

\begin{figure}
\begin{center}
\begin{tikzpicture}[scale=3]
  \draw (0,0) circle (1 cm);
  \foreach \x in {0,80,90,160,170,260,280,350} 
    \draw (0,0) -- (\x:1cm);
    \node at (40:1.15cm) {A}; 
     \node at (125:1.15cm) {C};
     \node at (200:1.15cm) {B};
     \node at (300:1.15cm) {D};
     \node at (270:1.15cm) {$D_1$};
     \node at (85:1.15cm) {$C_1$};
     \node[rotate=-30] at (60:1.15cm) {$\succeq$};
     \node[rotate=90] at (180:1.15cm){$\preceq$};
     \node at (165:1.25cm) {$C_2$};
\end{tikzpicture}
\end{center}
\caption{Two large pieces. \label{2pizza}}
\end{figure}
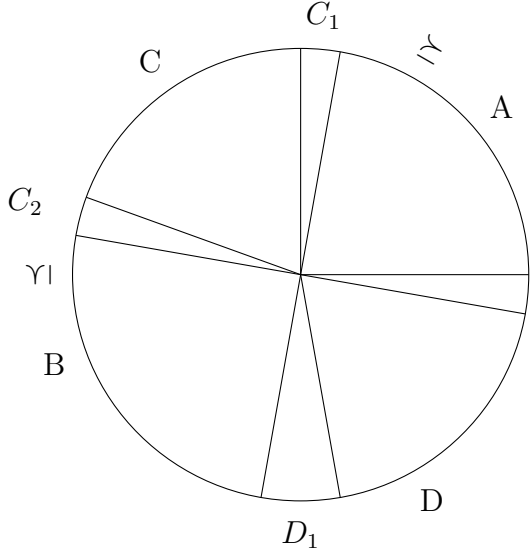

\textbf{CASE IV}: $b<d$ and $c>a$

Since $a, b>\frac{1}{5}$, $c+d<\frac{3}{5}$, which implies $c<\frac{3}{10}$. Therefore, if Alice starts by picking $C_1$ and Bob by picking $A$, when Bob finishes eating A, Alice at least eats as much as A in C. Let $\hat{C}$ denote the remaining pizza of C and the size of $\hat{C}$ is $\hat{c}$. $\hat{c} < \frac{3}{10}-\frac{1}{5}=\frac{1}{10}.$ Alice's strategy is to take pieces in $C$ if possible, until she gets to $B$. Against this strategy, Bob cannot eat pieces in $C$ as well after he finishes $A$ because even if he can get $B$, there is still the largest piece $D$ left, which means Bob cannot create least $\frac{1}{5}$ difference. The best thing Bob can do is to allow Alice to take all of $C$, but he takes pieces in $D$ and gets $B$ just before Alice does, hoping that the difference between $B$ and the remaining of $D$ is larger than $\frac{1}{5}$. However, this is impossible. The remaining of $D$, call it  $\hat{D}$ with size $\hat{d} > d-\hat{c} $  Hence, $b-\hat{d}< b-d+\hat{c}<\hat{d}=\frac{1}{10}<\frac{1}{5}$. Therefore, $C_1 \succeq A$. \\
If $b<c$, then use the same analysis, $C_2 \succeq B$. If $b>c$, $C_2 \succeq B$ as well by quoting the analysis from CASE II. Therefore, $C_2 \succeq B$ no matter what. The lemma applies.

\item If there are three big pieces. Label them $A$, $B$, $C$, and the remaining three clusters of pieces $D$, $E$, $F$. They have size $a$, $b$, $c$, $d$, $e$, $f$ respectively. Consider figure \ref{3pizza}.

Assume Bob can take at least $\frac{3}{5}$ of the whole pizza regardeless of what Alice does, then all the adjacent pieces must have uniform preference order. Let $F_1$ and $E_1$ be the pieces in $F$ and $E$ closest to $A$ and $F_n$ in $F$ the closest to $C$. WLOG, assume $F_1 \prec A$ and $A \prec E_1$. Let Alice start by picking $A$ and Bob $E_1$. When Alice finishes eating $A$, one of the following two situations must occur:
\begin{enumerate}
 \item Bob is still eating some small pieces from $E$ or $F$ and hasn't touch either $B$ or $C$;
 \item Bob eats the entire $F$, i.e. $F_1$ to $F_n$, and gets to $C$.
\end{enumerate}

Note that $a+b+c>\frac{3}{5}$, so $d+e+f<\frac{2}{5}$. If situation 1 is true, then when Alice finishes piece $A$, the remaining size of small pieces of $D$, $E$, $F$ is less than $\frac{1}{5}$, which makes it impossible for Bob to get both of $B$ and $C$. Even though Bob manages to get one of them, while he is eating it, Alice has enough time to claim the last big piece because the each large piece is bigger than $\frac{1}{5}$ but the total size of remaining small pieces is smaller than it. Therefore, Bob cannot win if he wastes his time eating small pieces. 
 
If Bob is eating some big piece, then it cannot be piece $B$. In fact, Bob cannot even touch piece in $F$. Consider the small pieces Bob eats before he takes the big one: $E_1...E_k, F_1...F_m$ (Not necessary representing the order he takes). Since Bob has time to get to the big piece, the size of these small piece must be strictly smaller than $a$. Therefore, it doesn't matter which order Bob taking them. In particular, it makes no difference for Bob to start from $E_1$ or $F_1$. However, recall that agianst Alice picking $A$, Bob picking $E_1$ is a winning strategy, and $F_1$ is a losing one. Therefore, before he gets the large piece, Bob cannot take $F_1$ at all. Otherwise $F_1$ and $E_1$ will be the same strategy against $A$. This means that the big piece Bob takes must be $C$.  

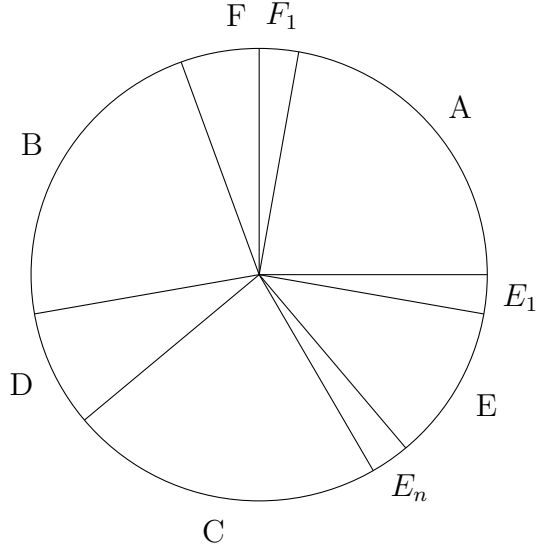
\begin{figure}
\begin{center}
\begin{tikzpicture}[scale=3]
  \draw (0,0) circle (1 cm);
  \foreach \x in {0,80,90,110,190,220,300,310,350} 
     \draw (0,0) -- (\x:1cm);
     \node at (40:1.15cm) {A}; 
     \node at (95:1.15cm) {F};
     \node at (150:1.15cm) {B};
     \node at (205:1.15cm) {D};
     \node at (260:1.15cm) {C};
     \node at (330:1.15cm) {E};
     \node at (85:1.15cm)  {$F_1$};
     \node at (355:1.15cm) {$E_1$};
     \node at (305:1.15cm) {$E_n$};
\end{tikzpicture}
\end{center}
\caption{Three large pieces. \label{3pizza}}
\end{figure}
In order to prevent Alice from getting the last large piece, it is mecessary that $e+c<a+f$ and $e+C<a+d$ because $A$,$F$ and $A$,$D$ are two possible ways for Alice to get to $B$. 

Then suppose Alice starts by picking $B$ and $C$. Again, when Alice finishes her first piece, situation 1 can't happen. Situation 2 is the only hope for Bob. Using the same analysis, we will get four other inequalities: $f+a<b+e$, $f+a<b+d$, $d+b<c+f$ and $d+b<c+e$. Once sum up these six inequalities, each letter occurs exactly twice on both sides of the inequality, so we get $0<0$. Contradiction.

What if two large pieces are adjacent, i.e. $def=0$? WLOG, assume $d=0$ so that $B$ and $C$ are adjacent. I want to show that for any configuration of finally many pieces of pizza, I am allowed to insert a sufficently small extra piece between any two pieces so that the outcome of the game doesn't change. If this is true, then I can insert a small piece between $B$ and $C$ and use the previous result. 
Since the number of pieces of the pizza is finite, label them $P_1,...P_n$. Let $\mathcal{S}=\{ x \mid x=\sum_{i=1}^n c_iP_i, c_i=-1,0,1 \}$. Therefore, $\mathcal{S}$ is all the possible differences a player faces when he or she is making a decision on what to pick next. Clearly $\mathcal{S}$ is finite and there exists a minumum $m$. Let $\epsilon <m$, so wherever I insert a $\epsilon$ size piece between two pieces, the final result won't be changed.

\item If there are four big pieces, then Alice can get at least $\frac{2}{5}$ by picking the smallest one of them. If Bob wants to win, he must take three of the four large pieces, which means that he must finish eating two large pieces (Call them $A$ and $B$ with size $a$ and $b$ respectively) and two small pieces (call them $D$ and $E$ with size $d$ and $e$ respectively) before Alice finishes one large piece $C$ and one small piece $F$ with size $c$ and $f$. Note that each big piece is bigger than $\frac{1}{5}$, so the sum of four small pieces (which can be 0) is less than $\frac{1}{5}$. Then, $a>c$($C$ is the smallest among the four big pieces) and $b>f$, which implies $a+b+d+e>a+b>c+f$. Hence it is impossible for Bob to eat $A$, $B$, $D$ and $E$ before Alice eats $C$ and $F$.
\end{itemize} 
\end{proof}

\begin{center}
\section{further question}
\end{center}

One can generalize this game to arbitrary graph. Given a graph, two players delete vertices with weights but they must have the graph remain connected. One needs x units of time to remove a vertex with weight x. What's the proportion of weights the starting player is guaranteed to get irrespective of the configuration? In this general setting, the pizzza problem is a special problem where the graph is restricted to cycle. One can ask the same question to pathes, trees and any other kind of graph. 

However, the problem becomes trivial if the graph is a path. In order to have the graph remain connected, Alice can only pick one of the two ends. Hence, it is possible hide a vertex with large vertex in the middle. For example, consider a path with three vertices, with weights 1, 100, 1 from left to right. Since Alice is only allowed to pick an end, the vertex with weight 100 is immediately available for Bob after Alice's first move. Therefore, Bob can get arbitrary large proportion as long as the graph favors him. 

On the other hand, the case for tree is not trivial at all if the tree is not a path. In this setting, each player is only allowed to delete a leaf of the tree, so it's no use to hide a vertex with large weight at root because players cannot remove it if leaves are still available. Even the simplest non trivial three involves complicated analysis. Consider Figure \ref{tree}. The fact that $A$ is not immediately available for Alice makes it much harder than the case of a cycle with four vertices with arbitrary size. It can be easily shown that Alice is able to get the vertex with the largest weight, so $\alpha > \frac{1}{4}$. In order to find the best possible bound, we can consider very ugly case analysis but it can't be generalized to arbitrary graph. Future efforts are required to find some more general or abstract approaches other than case analysis to shed light on this intriguing problem.

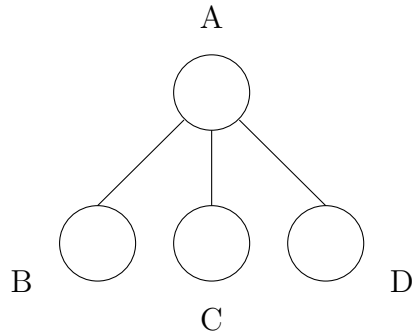
\begin{figure}
\begin{center}
\begin{tikzpicture}[scale=0.5]
  \draw (0,0) circle (1cm);
  \draw (3,-4) circle (1cm);
  \draw (-3,-4) circle (1cm);
  \draw (0,-4) circle (1cm);
  \draw (0.727,-0.727) -- (3,-3);
  \draw (-0.727,-0.727) -- (-3,-3);
  \draw (0,-1) -- (0,-3);
  \node at (0,2) {A};
  \node at (-5,-5) {B};
  \node at (0,-6) {C};
  \node at (5,-5) {D};
\end{tikzpicture}
\end{center}
\caption{A simple tree. \label{tree}}
\end{figure}

\textbf{Acknowledgement}. I am obliged to Wesley Pegden for his offering to be my research advisor. 

\begin{center}

\end{center}

\end{document}